\documentclass{article}

\usepackage{tikz}
    \usetikzlibrary{positioning}
    \tikzset{main node/.style={circle,fill=blue!25,draw,minimum size=.5cm,inner sep=0pt}}
    \usetikzlibrary{cd}
\usepackage{amsmath,amsfonts,amssymb,amsthm,mathtools,pdflscape}
\usepackage{etaremune, enumerate, float, verbatim, booktabs, comment, enumitem, authblk}
    \newtheorem{theorem}{Theorem}[section]
    \newtheorem{proposition}[theorem]{Proposition}
       
    \newtheorem{corollary}[theorem]{Corollary}
    \newtheorem{question}[theorem]{Question}
    \newtheorem{observation}[theorem]{Observation}
    \newtheorem{lemma}[theorem]{Lemma}
    
    \newtheorem{conjecture}[theorem]{Conjecture}
    \newtheorem{example}[theorem]{Example}
    \theoremstyle{definition}
        \newtheorem{definition}[theorem]{Definition}
    \setlistdepth{9}
\usepackage[margin=1.25in]{geometry}
\usepackage[capitalise,nameinlink]{cleveref}

\newcommand{\zbar}{\overline{Z}(G)}
\newcommand{\zbarn}{\overline{Z}(G_n)}

\newcommand{\cl}{\operatorname{cl}}
\newcommand{\Term}{\operatorname{Term}}
\newcommand{\PT}{\operatorname{PT}}
\newcommand{\pt}{\operatorname{pt}}
\newcommand{\md}{\operatorname{md}}
\newcommand{\pd}{\operatorname{pd}}

\title{Maximizing the discrepancy between zero forcing parameters relative to graph order}

\author{Andrew McKay}
\author{Allie Ray}
\author{Jonathan Valliere}
\affil{Department of Mathematics and Computer Science, Wheaton College, IL}
\date{\vspace{-8pt}\today}

\begin{document}
\maketitle

\begin{abstract}\label{abstract}
    \noindent Zero forcing is a process described by a color change rule on the vertices of a graph. In this paper, we maximize the discrepancy between various zero forcing parameters relative to graph order. First, we find an upper bound on the difference in cardinality between minimal zero forcing sets (sets containing no proper zero forcing subset) of maximum and minimum size, and we show that this bound is sharp for an infinite family of graphs. Furthermore, we derive an upper bound for the discrepancy between the maximum and minimum propagation times of the minimum zero forcing sets of any graph, showing this bound is sharp for an infinite family of graphs.
\end{abstract}

 \vspace{3pt}


\section{Introduction}\label{intro}
Consider a simple, undirected graph $G=(V,E)$, where every vertex starts out either colored blue or uncolored. The \textit{zero forcing color change rule} states that, at every iteration, a blue vertex adjacent to exactly one uncolored vertex colors its neighbor blue, which we term forcing. If an initial set of blue vertices is able to color each vertex in the graph through repeated applications of the standard zero forcing color change rule, then the set is called a \textit{zero forcing set}, or \textit{Z-set}. The \textit{zero forcing number}, denoted $Z(G)$, is the minimum number of vertices needed to be colored blue initially such that each vertex in the graph will be colored blue. This parameter was first introduced in \cite{AIM} to bound the minimum rank of symmetric matrices whose off-diagonal non-zero pattern follows that of the adjacency matrix of a graph, signaling strong ties to linear algebra. Zero forcing also has applications in quantum theory, examining how the manipulation of a local particle group can impact an entire quantum network \cite{burgarth}. It is closely related to the concept of graph domination, used to efficiently monitor an entire electrical system \cite{haynes}, and has applications in computer science, such as through the fast-mixed network search model introduced in \cite{yang}.

In recent years, the study of zero forcing has been greatly expanded and new parameters have been introduced. One such parameter is propagation time, first formally introduced in \cite{chilakamarri} but significantly expanded upon in \cite{hogben}, effectively shifting the study of zero forcing from merely a question of set size to a temporal examination of the process. Propagation time of a Z-set is the number of iterations it takes for the set to completely color a graph. Building on this, the concept of throttling was subsequently introduced in \cite{butler} to minimize the sum of the size of a Z-set and its propagation time. Variations on the standard zero forcing color change rule have also been introduced, and current research now includes other color change rules such as skew forcing and positive semidefinite forcing \cite{ima, barioli}. This paper, though, deals exclusively with standard zero forcing.

Finding and analyzing minimal zero forcing sets, defined as Z-sets that contain no proper zero forcing subset, is among the most recent areas of study related to zero forcing. Minimal Z-sets were first introduced in \cite{brimkov} in an attempt to better understand the zero forcing process, specifically by examining the relationship between minimum Z-sets and minimal Z-sets of maximum size. In \cite{brimkov}, the maximum cardinality of a minimal Z-set is defined as $\zbar$. We answer an open question presented in \cite{brimkov} by maximizing the discrepancy, $\zbar - Z(G)$, relative to the order of a graph. Specifically, we find that for any graph $G$ of order $n\geq 17$, the maximum discrepancy is $\zbar - Z(G) = n-6$ (see Theorem~\ref{one theorem to rule them all}). We also apply some of the same strategies used in answering this question to bound the difference between \textit{maximum propagation time}, $\PT(G)$, and \textit{minimum propagation time}, $\pt(G)$, over minimum Z-sets. In \cite{hogben}, the relationship between these two parameters is briefly examined through \textit{propagation time intervals} and \textit{propagation time discrepancy}, though we extend this research by sharply bounding the propagation time discrepancy for any graph relative to its order. We show that for any graph $G$, $\PT(G) - \pt(G) \leq \left\lfloor n-2\sqrt{n}+1 \right\rfloor$ (see Theorem~\ref{pd upper bound}).

This paper is structured as follows. In Section \ref{prelims}, we more precisely define some graph theoretic and zero forcing notation that is essential to this paper. In Section \ref{zbar-z}, we bound the discrepancy between $\zbar$ and $Z(G)$ with respect to graph order and show that this bound is sharp for an infinite family of graphs. Next in Section \ref{pd}, we find an upper bound on the discrepancy between $\PT(G)$ and $\pt(G)$ for all graphs, showing that this bound is sharp for an infinite family of graphs. Finally in Section \ref{conclusion}, we conclude with some open questions for future research.

\section{Preliminaries}\label{prelims}

\subsection{Graph theory terminology}
A simple undirected graph $G = (V,E)$ has a vertex set $V(G)$ and an edge set $E(G)$, where $E(G)$ consists of unordered pairs of elements in $V(G)$. 
The \textit{order} of a graph $G$ is $n=|V(G)|$, and two vertices $v,w \in V(G)$ are \textit{adjacent} if $vw\in E(G)$. The \textit{degree} of a vertex $v$, denoted deg($v$), is the number of vertices to which it is adjacent, and the \textit{minimum degree} of a graph is ${\delta(G) = \text{min}\{deg(v) \; | \; v\in V(G)\}}$. A vertex $v$ is called a \textit{leaf} if deg($v$) $= 1$. 
The \textit{open neighborhood} of a vertex $v$, denoted $N(v)$, is the set consisting of all vertices to which $v$ is adjacent, while the \textit{closed neighborhood} of $v$ is defined $N[v] = \{v\} \cup N(v)$. 

Various operations allow us to obtain new graphs by combining existing ones. The \textit{union} of two graphs, $G$ and $G'$, is denoted $G \cup G'$. The \textit{join} of two graphs $G$ and $G'$, denoted $G \vee G'$, is obtained by adding an edge between every vertex $v \in V(G)$ and $w \in V(G')$ in $G \cup G'$. 
The \emph{complement} of a graph $G$, denoted $\overline{G}$, is composed of vertices $V(G)$ and edges $\{vw \; | \; v,w\in V(G) \text{ and } vw\notin E(G)\}$.
Two graphs $G \text{ and } H$ are \textit{isomorphic} if there exists a bijection $\varphi: V(G) \rightarrow V(H)$ such that $vw\in E(G)$ if and only if $\varphi(v)\varphi(w)\in E(H)$. If $G \text{ and } H$ are isomorphic, we say $G \cong H$. 
A \textit{subgraph} $G'$ of a graph $G$ has $V(G') \subseteq V(G)$ and $E(G') \subseteq E(G)$. An \textit{induced subgraph} $G'$ of a graph $G$ has $V(G') \subseteq V(G)$ and $E(G') \subseteq E(G)$ where $E(G') = \{wv \; | \; w,v\in V(G') \text{ and } wv\in E(G)\}$. 

The following are characteristics of graphs that will be of interest in this paper. A graph $G$ is \textit{outerplanar} if it can be embedded in a plane such that no edges cross and all vertices are incident to the exterior, unbounded face \cite{barioli}. In \cite[Definition 3.1]{johnson}, a graph $G$ is defined to be a \textit{graph of two parallel paths} if there exist two independent induced paths of G that cover all the vertices of G and such that any edges between the two paths can be drawn so as to not cross. This excludes a simple path, but does include a graph of two disjoint paths. By this definition, a graph on two parallel paths is outerplanar.

We use several common families of graphs as examples throughout this paper and summarize their definitions and notations here. For the following graphs of order $n$, we denote a \textit{path} as $P_n$, a \textit{cycle} as $C_n$, a \textit{complete bipartite graph} as $K_{r,s}$, where $r+s=n$, a \textit{complete graph} as $K_n$, and an \textit{empty graph} as $\overline{K}_n$. A \textit{tree} is a connected acyclic graph, and a \textit{branch point} $v$ of a tree is a vertex of degree greater than 2. A \textit{spider}, denoted $S_{a_1,a_2,...a_k}$, is a tree with exactly one branch point $v$, where $a_1,a_2,...a_k$ are the order of the connected components, called \textit{legs}, of $G-v$.

\subsection{Zero forcing terminology}

We continue by defining some terminology directly related to the zero forcing color change rule. In Section \ref{illustration}, we include an example to illustrate many of these definitions. The zero forcing color change rule states that, at every iteration, a blue vertex adjacent to exactly one uncolored vertex colors its neighbor blue, which we term \emph{forcing}. If a vertex $v$ forces $w$, we denote this as $v \rightarrow w$. Consider the following definitions on a graph $G$.

\begin{definition}
The \emph{closure} of a set $S \subseteq V(G)$, denoted $\cl(S)$, is the set of blue-colored vertices of $G$ after iteratively applying the zero force rule until further applications no longer color any new vertices. 
\end{definition}

Considering the upper bound on $\cl(S)$, we define the following.

\begin{definition}
    The set $S$ is a \emph{zero forcing set}, or \emph{Z-set}, if $\cl(S) = V(G)$. Otherwise, we call $S$ a \emph{failed zero forcing set}. Moreover, we define a \emph{minimal zero forcing set} to be a zero forcing set that contains no proper subsets which are themselves zero forcing sets.
\end{definition}

We define the following parameters based on the size of the zero forcing sets.

\begin{definition}
    A zero forcing set of minimum cardinality is a \emph{minimum zero forcing set}. The \emph{zero forcing number}, denoted $Z(G)$, is the cardinality of a minimum zero forcing set. The \emph{maximum size of a minimal zero forcing set} is denoted $\zbar$.
\end{definition}

Note that every minimum zero forcing set is minimal, but the converse is not true. We focus much of our work in this paper on the relationship between $Z(G)$ and $\zbar$, namely the difference between them. In essence, this is the discrepancy between minimal Z-sets of maximum and minimum size. We define this quantity as follows:

\begin{definition}\label{def2.4}
    The \emph{minimal zero forcing set discrepancy} of a graph is \[\md(G) = \zbar - Z(G).\]
\end{definition}

For more on minimal Z-sets and the relationship between $Z(G)$ and $\zbar$, see \cite{brimkov}. Next, we define some terms that are important when considering the zero forcing process.

\begin{definition} For a zero forcing set $S$ of a graph $G$, let $\mathcal{F}$ be the \emph{set of forces} of $S$.
\begin{itemize}
    \item A \emph{forcing chain} of $\mathcal{F}$ is a sequence of vertices $\{v_1, v_2, \dots, v_k\}$ such that $v_i$ forces $v_{i+1}$ for $i = 1, 2, \dots, k-1$. The \emph{length} of a forcing chain is $k-1$.
    \item A \emph{maximal forcing chain} is a forcing chain that is not a proper sub-chain of any other forcing chain. 
    \item The \emph{terminus} of $\mathcal{F}$, denoted $\Term(\mathcal{F})$, is composed of the vertices in $V(G)$ that do not perform a force in $\mathcal{F}$.
\end{itemize}
\end{definition}

Note that any set of forces of a graph $G$ will yield at least $Z(G)$ maximal forcing chains. Also using the notation above, note that the terminus of $\mathcal{F}$ includes the last vertex $v_k$ in any maximal forcing chain of $\mathcal{F}$.

Another way to examine the zero forcing process is to consider when certain vertices cannot be forced, seen in the following definition, and first introduced in \cite{fast}.

\begin{definition}
    A \emph{fort} is a non-empty set $F \subseteq V(G)$ such that no vertex $v \in V(G) \setminus F$ is adjacent to exactly one vertex in $F$.
\end{definition}

Note that for any subset of vertices $S \subseteq V(G)$, $\overline{\cl(S)}$ is a fort. Next, we define the following term that will be useful in choosing Z-sets to prove many of our results.

\begin{definition}
    The \emph{closed colored neighborhood} of a colored vertex $v$ under a zero forcing set $S$ is \[N_c[v] = N[v] \cap S.\]
\end{definition}

\begin{example}\label{ex1}
Let $G$ be the graph in Figure \ref{example1}, and define a Z-set $S=\{v_2,v_3,v_4\}$. Observe that $N_c[v_4]=\{v_3,v_4\}$.
\end{example}

\begin{figure}[H]
\begin{center}
\begin{tikzpicture}[scale=.65, transform shape]
    \node[main node, fill=white] (1) {$v_1$};
    \node[main node] (2) [right = 1.5cm of 1] {$v_2$};
    \node[main node] (3) [right = 1.5cm of 2] {$v_3$};
    \node[main node] (4) [below = 1.5cm of 3] {$v_4$};
    \node[main node, fill=white] (5) [below = 1.5cm of 2] {$v_5$};  
    
    \path[draw,thick]
    (1) edge node {} (2)
    (2) edge node {} (3)
    (3) edge node {} (4)
    (4) edge node {} (5)
    (5) edge node {} (1)
    (5) edge node {} (2)
    (5) edge node {} (3);
\end{tikzpicture}
\caption{A graph $G$ for Example \ref{ex1}.}
\label{example1}
\end{center}
\end{figure}
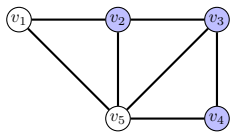

Finally, we define zero forcing parameters related to propagation.

\begin{definition}
The \emph{propagation time} of a zero forcing set is the number of iterations of the zero forcing rule needed to color the entire graph.
\begin{itemize}
    \item The \emph{minimum propagation time}, denoted $\pt(G)$, is the smallest propagation time across all minimum zero forcing sets.
    \item The \emph{maximum propagation time}, denoted $\PT(G)$, is the largest propagation time across all minimum zero forcing sets.
\end{itemize}
\end{definition}

Note that the propagation time of the set $V(G)$ is 0. The parameter representing the difference between these two quantities was first established in \cite{hogben} and was defined as follows.

\begin{definition}
    The \emph{propagation time discrepancy} of a graph is \[\pd(G) = \PT(G) - \pt(G).\]
\end{definition}

For more on zero forcing propagation time, see \cite{hogben}.

\subsection{Illustrating parameters on an example graph}\label{illustration}

To illustrate some of these definitions, we will examine various parameters on an example graph $G$ (see Figure \ref{example}).

\begin{figure}[H]
\begin{center}
\begin{tikzpicture}[scale=.65, transform shape]
    \node[main node, fill=white] (1) {$v_1$};
    \node[main node, fill=white] (2) [right = 1cm of 1] {$v_2$};
    \node[main node, fill=white] (3) [right = 1cm of 2] {$v_3$};
    \node[main node, fill=white] (4) [right = 1.3cm of 3] {$v_4$};
    \node[main node, fill=white] (5) [right = 2.3cm of 4] {$v_5$};   
    \node[main node, fill=white] (6) [above right = 0.65cm and 0.10cm of 4] {$v_6$};
    \node[main node, fill=white] (7) [above = 1.5cm of 3] {$v_7$};
    
    \path[draw,thick]
    (1) edge node {} (2)
    (2) edge node {} (3)
    (3) edge node {} (4)
    (4) edge node {} (5)
    (5) edge node {} (6)
    (6) edge node {} (7)
    (3) edge node {} (7)
    (4) edge node {} (7);   
\end{tikzpicture}
\caption{Graph $G$ for Section \ref{illustration}.}
\label{example}
\end{center}
\end{figure}
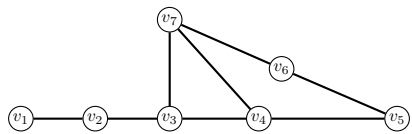

The parameters and their corresponding values are listed below in Table \ref{params}. Also included for each relevant parameter is an example set of vertices of $G$ that yields the listed value.

\begin{table}[H]
    \centering
    \caption{Parameters and example yielding sets of $G$.}
    \vspace{8pt}
    \label{params}
    \begin{tabular}{l l l}
        \toprule
        \textbf{Parameter} & \textbf{Value} & \textbf{Yielding Set} \\
        \midrule
        $Z(G)$ & 2 & $\{v_1, v_7\}$ \\
        $\zbar$ & 3 & $\{v_3, v_4, v_7\}$ \\
        $\md(G)$ & 1 & --- \\
        $\pt(G)$ & 4 & $\{v_1, v_4\}$ \\
        $\PT(G)$ & 5 & $\{v_4, v_5\}$ \\
        $\pd(G)$ & 1 & --- \\
        \bottomrule
    \end{tabular}
\end{table}

First observe that the Z-set $\{v_1, v_7\}$ yields the set of forces $\{v_1 \rightarrow v_2, v_2 \rightarrow v_3, v_3 \rightarrow v_4, v_4 \rightarrow v_5, v_7 \rightarrow v_6\}$, where the last two forces occur at the same time step. The forcing chains, $F_1 \ \text{and} \ F_2$, are then $\{v_1, v_2, v_3, v_4, v_5\} \ \text{and} \ \{v_7, v_6\}$, respectively. Note that the yielding sets in Table \ref{params} are not necessarily unique. For example, the sets $\{v_1, v_7\} \text{ and } \{v_6,v_7\}$ are two minimum Z-sets of $G$ among many. Additionally, see that not all minimal Z-sets of a graph are necessarily of the same cardinality, and that minimum Z-sets can yield a range of propagation times.

\section{Maximizing minimal zero forcing set discrepancy}\label{zbar-z}

 In this section, we seek to answer a question posed in \cite[Question 4.3]{brimkov}: What is the largest possible gap between $\zbar$ and $Z(G)$? Recall from Definition \ref{def2.4} that we term this gap to be the minimal zero forcing set discrepancy, defined as $\md(G)=\zbar-Z(G)$. In \cite{brimkov}, an infinite family of graphs that each satisfy $\md(G) = n-7$ is presented. It is also clear to see that any graph $G$ satisfies $\md(G) \leq n-1$, as $\zbar \leq n$ and $Z(G) \geq 1$. We further investigate the maximum size of this gap, and will prove the following results:
\begin{itemize}
    \item The minimal zero forcing set discrepancy can be arbitrarily large. For any nonnegative integer $x$, there is a graph that yields $\md(G) = x$.
    \item Minimal zero forcing set discrepancies up to $n-1$ can be achieved, but only by certain graphs of small order.
    \item When restricted to graphs of order $n\geq 17$, the maximum value of the minimal zero forcing set discrepancy as a function of the order is $\md(G) = n-6$.
\end{itemize}

The first result we will demonstrate in Corollary \ref{allmd}, and the second we will briefly examine in Observation \ref{finitefam}. The third result is of most significance and the one to which we will dedicate most of this section. It prompts the exploration of graphs of arbitrarily large order that have a large minimal Z-set discrepancy relative to $n$. We first introduce a proposition that will be used towards this end.

\begin{proposition}\label{obs}
    If $G$ is a graph of order $n$ satisfying $\zbar \geq n-1$, then $\md(G) = 0$.
\end{proposition}

\begin{proof}
    If $\zbar = n$, then $G\cong \overline{K}_n$ by \cite[Observation 2.1]{brimkov}. Furthermore, if $\zbar = n-1$, then $G \cong K_m \cup kK_1$ by \cite[Proposition 2.2]{brimkov}, where $k$ is a nonnegative integer and $m = n-k \geq 2$. In either case, by \cite[Corollary 2.4]{brimkov} and the fact that every Z-set must contain all isolated points, we see that $Z(G) = \zbar$ and thus $\md(G) = 0$
\end{proof}

As mentioned previously, there exist graphs of small order that achieve high minimal set discrepancies relative to graph order. Although these are of little significance to our primary investigation, a brief observation may prove useful in understanding these edge cases.

\begin{observation}\label{finitefam}
By Proposition \ref{obs}, we note that there exist graphs of small order that satisfy $\md(G)=0$ and thus trivially yield $\md(G) > n-6$. Examples include $K_n$, $\overline{K}_n$, and $C_n$ for $n \leq 5$. The graph $K_1$, for example, satisfies $\md(K_1) = 0 = n-1$, a trivially high discrepancy relative to graph order. Also, note that paths of order $4 \leq n \leq 6$ yield $\md(G)=1 > n-6$. 
\end{observation}

In the remainder of Section \ref{zbar-z}, we turn our focus back towards graphs of arbitrarily large order, showing that $\md(G) = n-6$ is the maximum achievable discrepancy for these graphs.

\subsection{Lower bound on the maximum discrepancy relative to graph order}

An infinite family of graphs is established in \cite[Proposition 2.10]{brimkov} such that each graph $G$ in the family satisfies $\md(G) = n-7$. The family presents an initial lower bound on the maximum value of the minimal Z-set discrepancy. We first sharpen this bound by altering their example to produce an infinite family of graphs $\mathcal{G}$ such that each $G_n\in \mathcal{G}$ satisfies $\md(G_n) = n-6$.

\begin{proposition}\label{n-6 family}
    There exists an infinite family of graphs $\mathcal{G} = \{G_n\}$ such that $\md(G_n) = n-6$ for all order $n \geq 6$.
\end{proposition}

\begin{proof}
Let $G_n$ be the graph $(P_2 \cup K_1) \vee P_{n-3}$ for $n \geq 6$, with vertices labeled as in Figure \ref{n-6}.

\vspace{5pt}
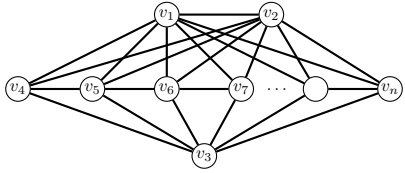
\begin{figure}[H]
\begin{center}
\begin{tikzpicture}[scale=.65, transform shape]
    \node[main node, fill=white] (1) {$v_4$};
    \node[main node, fill=white] (2) [right = 1cm of 1] {$v_5$};
    \node[main node, fill=white] (3) [right = 1cm of 2] {$v_6$};
    \node[main node, fill=white] (4) [right = 1cm of 3] {$v_7$};
    \node[main node, fill=white] (5) [right = 1cm of 4] {};   
    \node[main node, fill=white] (6) [right = 1cm of 5] {$v_n$};
    \node[main node, fill=white] (7) [above = 1cm of 3] {$v_1$};
    \node[main node, fill=white] (8) [above right = 1.15cm and 0.25cm of 4] {$v_2$};
    \node[main node, fill=white] (9) [below right = 1cm and 0.4cm of 3] {$v_3$};

    \path[draw,thick]
    (1) edge node {} (2)
    (2) edge node {} (3)
    (3) edge node {} (4)
    (5) edge node {} (6);
    \path[draw,thick]
    (7) edge node {} (1)
    (7) edge node {} (2)
    (7) edge node {} (3)
    (7) edge node {} (4)
    (7) edge node {} (5)
    (7) edge node {} (6)
    (8) edge node {} (1)
    (8) edge node {} (2)
    (8) edge node {} (3)
    (8) edge node {} (4)
    (8) edge node {} (5)
    (8) edge node {} (6)
    (9) edge node {} (1)
    (9) edge node {} (2)
    (9) edge node {} (3)
    (9) edge node {} (4)
    (9) edge node {} (5)
    (9) edge node {} (6)
    (7) edge node {} (8);
    \path (4) -- node[xshift=0cm] {$\dots$} (5);
\end{tikzpicture}
\caption{The graph $G_n$ satisfying $\md(G_n) = n-6$.}
\label{n-6}
\end{center}
\end{figure}

\vspace{-8pt}
First, observe that $S = \{v_1, v_2, v_3, v_4\}$ is a Z-set of size 4. It is a well known fact that $Z(G_n) \geq \delta(G_n)$, and thus when $n\geq 7$, $\delta(G_n) = 4$ and $S$ is a minimum Z-set. When $n=6$, $\delta(G_6) = 3$ and there is exactly one vertex with this degree, namely $v_3$. Thus, for a Z-set of size 3 to exist in $G_6$, it would need to be composed of $v_3$ and exactly two neighbors of $v_3$. But in $G_6$, $F = \{v_1, v_2\}$ is a fort and no vertices in $N[v_3]$ intersect $F$, thus no subset of $N[v_3]$ is a Z-set \cite[Proposition 2.2]{alameda}. Therefore $S$ is a minimum Z-set of size 4 for all $G_n\in \mathcal{G}$, and $Z(G_n) = 4$.

Consider also the Z-set $S' = \{v_1, v_4, v_5, v_6,... v_n\}$ of size $n-2$. Under this set, $v_1$ is the only vertex with exactly one uncolored neighbor, namely $v_2$. Each other vertex in $S'$ has two uncolored neighbors: $v_2$ and $v_3$.  Because each vertex in $\{v_3,v_4,...,v_n\}$ is adjacent to $v_1$, removing any vertex from $S'$ would result in a failed Z-set. Either $v_1$ would be uncolored, or it would have more than one uncolored neighbor and be unable to force. Thus, $S'$ is a minimal Z-set of size $n-2$.

Since $\md(G_n) > 0$ for $n\geq 7$, the contrapositive of Proposition \ref{obs} states that $\zbarn < n-1$ and thus $S'$ is indeed a minimal Z-set of maximum size. Similarly, when $n=6$, it can be seen that $G_6 \ncong \overline{K}_n$ and $G_6 \ncong K_m \cup kK_1$ where $k$ is a nonnegative integer and $m = n-k \geq 2$. Thus by the proof of Proposition \ref{obs}, $\overline{Z}(G_6) < n-1$ and $S'$ is a minimal Z-set of maximum size. Therefore $\zbarn = n-2$.

Thus for each $G_n\in \mathcal{G}$, \[\md(G_n) = \zbarn - Z(G_n) = n-6.\qedhere\]
\end{proof}

By producing this family, we also see that the value $\md(G_n)$ can be arbitrarily large.

\begin{corollary}\label{allmd}
    For any nonnegative integer $x$, there exists a graph $G_n$ for which $\md(G_n) = x$.
\end{corollary}

\begin{proof}
    Let $G_n$ be a graph in the family $\mathcal{G}$ from Proposition \ref{n-6 family}, and let $n=x+6$. Then,\[\md(G_{x+6}) = (x+6)-6 = x. \qedhere\]
\end{proof}

\subsection{Upper bound on the maximum discrepancy relative to graph order}

Though Proposition \ref{n-6 family} sharpens the lower bound, the question still remains: What is the largest possible value for minimal Z-set discrepancy? We will begin to show that the maximum value is in fact $n-6$ by first presenting an initial upper bound on the parameter, namely $n-4$.

\begin{proposition}\label{initial bound}
    If $G$ is a graph of order $n\geq 5$, then $\md(G) \leq n-4$.
\end{proposition}

\begin{proof}
If $\zbar \geq n-1$, then $\md(G) = 0 < n-4$ by Proposition \ref{obs}.  By definition, $Z(G) \geq 1$, and if $Z(G) = 1$, then $G$ is a path and $\md(G) \leq 2-1 = 1 \leq n-4$. Otherwise, $\zbar \leq n-2$ and $Z(G) \geq 2$, yielding $\md(G) \leq (n-2) - 2 = n-4$. Thus $\md(G) \leq n-4$ for all graphs of order $n\geq 5$.
\end{proof}

The proof of Proposition \ref{initial bound} implies that for a graph of high order to achieve a large minimal set discrepancy it must satisfy $Z(G)\geq 2$ and $\zbar \leq n-2$. 
Thus, there are only three possible combinations of values of $\zbar$ and $Z(G)$ that can produce a large order graph satisfying ${n-5 \leq \md(G) \leq n-4}$. They are as follows:
\begin{center}
\vspace{-2pt}
$Z(G) = 2$ and $\zbar = n-2$,

\vspace{3pt}

$Z(G) = 2$ and $\zbar = n-3$,

\vspace{3pt}

$Z(G) = 3$ and $\zbar = n-2$.
\end{center}

We will disprove each of these cases to show that ${n-5 \leq \md(G) \leq n-4}$ is unattainable for arbitrarily large values of $n$. We first present various techniques that will be repeatedly employed. The first two propositions are used to show that a chosen Z-set $S$ with certain properties is not minimal in $G$ by finding a proper subset of $S$ that is also a Z-set.

\begin{proposition}\label{reforce}
    Let $S$ be a zero forcing set of a connected graph $G$ of order $n \geq 2$. If there exists some vertex $v\in S$ that is adjacent to no vertices in $V(G)\setminus S$, then $S$ is not minimal.
\end{proposition}

\begin{proof}
    Assume $S$ is a Z-set. Let $S'= S \setminus y$ for some $y\in S\cap N(v)$. Then $v$ can force $y$, implying that $\cl(S') = \cl(S)$. This implies that $S'$ is a Z-set, and thus $S$ is not minimal.
\end{proof}

\begin{proposition}\label{nbhd}
    Let $S$ be a zero forcing set of a connected graph $G$, and let $\mathcal{Z}$ be a minimum zero forcing set. If there exists $S' \subset S$ satisfying $\mathcal{Z} \subseteq \cl(S')$, then $S$ is not minimal.
\end{proposition}

\begin{proof}
    If $\mathcal{Z} \subseteq \cl(S')$, then $\cl(S') = \cl(\mathcal{Z}) = V(G)$. Since we assume $S' \subset S$, $S$ is not minimal.
\end{proof}

While not directly contradicting the minimality of a Z-set, the following proposition is helpful in understanding the composition and forcing process of a Z-set.

\begin{proposition}\label{fort}
    If $S\subset V(G)$ is a zero forcing set of a graph $G$, then there exists some $v\in S$ that is adjacent to exactly one vertex in $V(G) \setminus S$.
\end{proposition}

\begin{proof}
    If $v$ does not exist, then $V(G) \setminus S$ is a fort. $S$ does not intersect this fort, and is therefore not a Z-set by \cite[Proposition 2.2]{alameda}, a contradiction.
\end{proof}

Since the structure of a graph relies heavily on the zero forcing number, we will first approach the cases where $Z(G) = 2$, before disproving the existence of a high order graph satisfying the third and final case.

\subsection{Bounding $\zbar$ when $Z(G)=2$}

For the case when $Z(G)=2$, we know the overall structure of the graph $G$ is a graph on two parallel paths by the following proposition in \cite{row}. We then consider all cases for the position of uncolored vertices in $G$ for a Z-set of size $n-3$, up to symmetry and isomorphism.
\begin{proposition}\cite[Theorem 2.3]{row}\label{Z2} Let $G=(V,E)$ be a graph. Then $Z(G)=2$ if and only if $G$ is a graph on two parallel paths. Furthermore, this means that $G$ will be outerplanar. \end{proposition}
The definition of a graph on parallel paths implies that the graph $G$ of order $n$ is of the form $P_y \cup P_z$, where $n=y+z$, with possible non-crossing edges connecting $P_y$ and $P_z$. For notation, we let $V(P_y)=\{v_1,v_2,...,v_y\}$ and $ V(P_z)=\{w_1,w_2,...,w_z\}$, noting that $\{v_j, w_k\}$ are pairwise distinct for any $j, k$. An example graph is seen in Figure \ref{parallel paths graph 4}. 
\vspace{15pt}
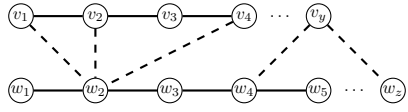
\begin{figure}[H]
\begin{center}
\begin{tikzpicture}[scale=.65, transform shape]
    \node[main node, fill=white] (1) {$w_1$};
    \node[main node, fill=white] (2) [right = 1cm of 1] {$w_2$};
    \node[main node, fill=white] (3) [right = 1cm of 2] {$w_3$};
    \node[main node, fill=white] (4) [right = 1cm of 3] {$w_4$};
    \node[main node, fill=white] (5) [right = 1cm of 4] {$w_5$};   
    \node[main node, fill=white] (6) [right = 1cm of 5] {$w_z$};
    \node[main node, fill=white] (7) [above = 1cm of 1] {$v_1$};
    \node[main node, fill=white] (8) [above = 1cm of 2] {$v_2$};
    \node[main node, fill=white] (9) [above = 1cm of 3] {$v_3$};
    \node[main node, fill=white] (10) [above = 1cm of 4] {$v_4$};
    \node[main node, fill=white] (11) [above = 1cm of 5] {$v_y$};
    \path[draw,thick]
    (1) edge node {} (2)
    (2) edge node {} (3)
    (3) edge node {} (4)
    (4) edge node {} (5)
    (7) edge node {} (8)
    (8) edge node {} (9)
    (9) edge node {} (10)
    (7) edge[dashed] node {} (2)
    (8) edge[dashed] node {} (2)
    (10) edge[dashed] node {} (2)
    (11) edge[dashed] node {} (4)
    (11) edge[dashed] node {} (6);
    \path (5) -- node[xshift=0cm] {$\dots$} (6);
    \path (10) -- node[xshift=0cm] {$\dots$} (11);
    \end{tikzpicture}
\caption{An example outerplanar graph of the form $P_y \cup P_z$.}
\label{parallel paths graph 4}
\end{center}
\end{figure}

We are primarily interested in the minimal zero forcing set discrepancy of graphs of arbitrarily high order. Thus, in many of our proofs we set a lower bound on the order to reduce the argument's complexity and number of cases, such as $n\geq 10$ in the following proposition.

\begin{proposition}\label{Side1}
If $G$ is a graph of order $n\geq 10$ that satisfies $Z(G)=2$, then $\zbar \neq n-3$.
\end{proposition}

\begin{proof}
Assume $G$ is a graph of order $n\geq 10$ such that $Z(G) = 2$. Because $Z(G) = 2$, we know that $G$ is a graph on parallel paths and outerplanar by \cite[Theorem 2.3]{row}. If $G$ is disconnected, then $G$ is two disjoint paths and $\zbar \leq 4 < n-3$. Thus, we will assume $G$ is connected and of the form $P_y \cup P_z$ with labeling as in Figure \ref{parallel paths graph 4}.

The parallel paths can be thought of as two maximal forcing chains of $G$ resulting from a minimum Z-set $\{v_1, w_1\}$. Observe also that $\{v_y, w_z\}$ is a minimum Z-set \cite[Theorem 2.6]{barioli}. Also, we define $\{v_1,w_1, v_y, w_z\}$ to be exterior vertices, and $V(G) \setminus \{v_1,w_1, v_y, w_z\}$ to be interior vertices.

For the sake of contradiction, let $S$ be a minimal Z-set of size $n-3$ and thus all vertices in $S$ are colored blue and the vertices in $V(G) \setminus S$ are uncolored. We now consider all possible characterizations of $V(G) \setminus S$, up to isomorphism and symmetry of the graph, each contradicting $S$ being a minimal Z-set.

\begin{description}
    \item[Case 1:]  Suppose at least two of the vertices in  $V(G) \setminus S$ are interior vertices.
    
    Then, either $\{v_1, w_1\} \subset S$ or $\{v_y, w_z\} \subset S$, both of which are minimum Z-sets.

    \item[Case 2:]  Suppose two of the vertices in  $V(G) \setminus S$ are on the same end, and the third is an interior vertex. Without loss of generality, let $\{v_1, w_1\} \subset V(G) \setminus S$.

    Then, the minimum Z-set $\{v_y, w_z\} \subset S$.

    \item[Case 3:]  Suppose two of the vertices in $V(G)\setminus S$ are on opposite ends and on different paths, and the third is an interior vertex. Without loss of generality, say $V(G) \setminus S = \{v_1, w_z, v_j\}$, where $2\leq j \leq y-1$.

    \begin{description}
        \item[Case a:]  Let $z \geq 5$.

        Consider the set $W=\{w_1,...,w_{z-2}\}$. If there exists $w_i \in W$ that is not adjacent to either $\{v_1,v_j\}$, then apply Proposition \ref{reforce} to $w_i$. Otherwise, suppose there exists $w_i \in W$ that is adjacent to $v_1$ and not $v_j$. Note that $\{w_1,..., w_{i-1}\}$ are not adjacent to $v_j$ by outerplanarity, thus Proposition \ref{nbhd} is satisfied by choosing $S' = N_c[w_1]$ and $\mathcal{Z} = \{v_1, w_1\}$. If neither of these exist, then each vertex $w_i \in W$ is adjacent to $v_j$. Proposition \ref{nbhd} is satisfied by choosing $S' = S \setminus \{w_1\}$ and $\mathcal{Z} = \{v_y, w_z\}$, because $w_3 \rightarrow v_j$ since $w_3 \notin N(v_1) \cup N(w_z)$. Then $w_{z-1} \rightarrow w_z$ because $w_{z-1} \notin N(v_1)$, and thus $\mathcal{Z} \subset \cl(S')$.
        
        \item[Case b:]  Let $z<5$.

        Then $y \geq 6$ because $y+z \geq 10$. Let $X=\{v_i\; | \; v_i \notin N[v_1] \cup N[v_j]\}$, and note that $X \neq \emptyset$. If there exists some $v_i \in X$ such that $v_i \notin N(w_z)$, then apply Proposition \ref{reforce} to $v_i$. Otherwise, each $v_i \in X$ is adjacent to $w_z$. Note that since at least four vertices of $P_y$ are colored blue because $y \geq 6$, there will always be some $v_k\notin \{v_1, v_j\}$ such that $v_k \notin N[v_i]$. Let $S' = S \setminus \{v_k\}$. Then $v_i \rightarrow w_z$, and $P_z$ is now fully colored blue. Next, because $\{v_{i-1},v_i\}$ are colored blue and adjacent on $P_y$, then $P_y$ will be forced. Thus $S$ is not minimal.

    \end{description}

    \item[Case 4:]  Suppose two of the vertices of $V(G)\setminus S$ are on opposite ends of the same path subgraph, and the third is an interior vertex on that path. Without loss of generality, say $V(G_n) \setminus S = \{v_1, v_j, v_y\}$, where $2\leq j \leq y-1$.
    
    \begin{description}
        \item[Case a:]  Let $z \geq 5$.
    
        Define the set $W=V(P_z)$. If there exists some $w_i\in W$ that is not adjacent to $v_1,v_j, \text{ or } v_y$, then apply Proposition \ref{reforce} to $w_i$. If there exists some $w_i\in W$ that is only adjacent to one uncolored vertex, and that is $v_1 \text{ or } v_y$, without loss of generality say $v_1 \in N[w_i]$, then Proposition \ref{nbhd} is satisfied by choosing $S' = N_c[w_i]\cup w_1$ and $\mathcal{Z} = \{v_1, w_1\}$. 
        
        If there exists some $w_i$ adjacent to both $v_1$ and $v_y$, but not $v_j$, there would exist $w_\ell \neq w_i$ that would be adjacent to either one or no vertex from the set $\{v_1, v_y\}$ by outerplanarity, and these have already been considered in the previous paragraph. This leaves the case when each $w_i \in W$ is adjacent to $v_j$. So, let $S' = S \setminus \{w_1\}$. Then $w_3 \rightarrow v_j \text{ and } w_2 \rightarrow w_1$ because the edge $v_jw_1$ exists. Thus $S\cup \{v_j\} \subseteq \cl(S')$. Because $S$ is a Z-set, $\cl(S\cup \{v_j\}) = V(G)$ and thus $\cl(S') = V(G)$ and $S$ is not minimal.

        \item[Case b:]  Let $z < 5$.

        Then $y \geq 6$. Note that $\{v_2,...,v_{y-1}\}$ contains at least 4 vertices, and all but one of these are colored blue. Thus, there always exists two adjacent vertices in $\{v_2,...,v_{y-1}\} \cap S$, call them $\{v_i, v_{i+1}\}$. Let $S' = V(P_z) \cup \{v_i, v_{i+1}\}$. Then $P_y$ will be forced, and so $S'\subset S$ is a Z-set and $S$ is not minimal.
    \end{description}

    \item[Case 5:]  Suppose two of the vertices of $V(G)\setminus S$ are on opposite ends of the same path subgraph, and the third is an interior vertex not on that path. Without loss of generality, say $V(G) \setminus S = \{v_1, v_y, w_k\}$, where $2\leq k \leq z-1$.
    
    \begin{description}
        \item[Case a:]  Let $z \geq 5$.

        Consider the set $W = V(P_z) \setminus N[w_k]$. If there exists $w_i \notin  W$ adjacent to neither $v_1$ nor $v_y$, then apply Proposition \ref{reforce} to $w_i$. Note that $|W| \geq 2$, and thus there exists some $w_i \in W$ that is not adjacent to both $\{v_1, v_y\}$ by outerplanarity. Without loss of generality, let $w_i$ be adjacent to $v_1$ and not $v_y$. Then Proposition \ref{nbhd} is satisfied by choosing $S' = N_c[w_i] \cup \{w_1\}$ and $\mathcal{Z} = \{v_1,w_1\}$.

       \item[Case b:]  Let $3 \leq z < 5$.

        Then $y \geq 6$. Let $X=\{v_3,...,v_{y-2}\}$. If there exists some $v_i \in X$ that is not adjacent to $w_k$, then apply Proposition \ref{reforce} to $v_i$. Otherwise, each $v_i \in X$ is adjacent to $w_k$. Then Proposition \ref{nbhd} is satisfied by choosing $S' = S \setminus \{w_z\}$ and $\mathcal{Z} = \{v_1, w_1\}$. Because $G$ is outerplanar and the edge $v_3w_k$ exists, $v_3 \rightarrow w_k$, and then $v_2 \rightarrow v_1$, thus $\mathcal{Z} \subseteq \cl(S')$.
    \end{description}

    \item[Case 6]:  Suppose the vertices of $V(G)\setminus S$ are all endpoints. Without loss of generality, say $V(G) \setminus S = \{v_1, v_y, w_z\}$. 

    \begin{description}
        \item[Case a:]  Let $z \geq 5$.

        Define the set $W=\{w_1,...,w_{z-2}\}$. If there exists $w_i \in W$ adjacent to neither $v_1$ nor $v_y$, then apply Proposition \ref{reforce} to $w_i$. Otherwise, suppose there exists $w_i \in W$ that is adjacent to $v_1$ and not $v_y$. Note that $\{w_1,..., w_{i-1}\}$ are not adjacent to $v_y$ by outerplanarity, thus Proposition \ref{nbhd} is satisfied by choosing $S' = N_c[w_1]$ and $\mathcal{Z} = \{v_1, w_1\}$. If neither of these exists, then each $w_i \in W$ is adjacent to $v_y$. Then Proposition \ref{nbhd} is satisfied by choosing $S'=N_c[w_2]\cup N_c[w_{z-1}]$ and $\mathcal{Z} = \{v_y, w_z\}$. Since $w_2 \rightarrow v_y$, and then $w_{z-1} \rightarrow w_z$, we have $\mathcal{Z} \subseteq \cl(S')$. Furthermore, $S' \subset S$ since $|S'| \leq 5 < n-3 = |S|$.

        \item[Case b:]  Let $z<5$. 

        Then $y \geq 6$. Let $X=\{v_3,...,v_{y-2}\}$. If there exists $v_i \in X$ not adjacent to $w_z$, then apply Proposition \ref{reforce} to $v_i$. Otherwise, each $v_i \in X$ is adjacent to $w_z$. And then Proposition \ref{nbhd} is satisfied by choosing $S' = N_c[v_4]\cup N_c[v_{y-1}]$ and $\mathcal{Z} = \{v_y, w_z\}$. Because the edge $v_4w_z$ exists, $v_4 \rightarrow w_z$, and then $v_{y-1} \rightarrow v_y$, thus $\mathcal{Z} \subset \cl(S')$. Note that $v_2 \notin S'$, and so $S' \subset S$.
    \end{description}

\end{description}

In every case, we see that if $S$ is a Z-set such that $|S| = n-3$, then $S$ is not minimal. Thus $\zbar \neq n-3$.
\end{proof}

Furthermore, we can now use this result to show that $\zbar < n-3$.

\begin{theorem}\label{cor2}
     If $G$ is a graph of order $n \geq 10$ satisfying $Z(G) \leq 2$, then $\zbar < n-3$.
\end{theorem}

\begin{proof}

If $Z(G)=1$, then $G$ is a path. By definition, paths of order, $n \geq 10$ have $\zbar=2 < n-3$. Consider when $Z(G)=2$. If $G$ is disconnected, it is two disjoint paths and $\zbar \leq 4 < n-3$. Thus, we will assume $G$ is connected. Note that by Proposition \ref{initial bound}, when $Z(G) = 2$ and $n\geq 10$, we know that $\zbar \leq n-2$. We now will show that $\zbar \neq n-2$. Consider, for contradiction, a graph $G$ where $Z(G) = 2 \text{ and }\zbar = n-2$ using the same labeling as Figure \ref{parallel paths graph 4}. Let $S$ be a Z-set of size $n-2$. If $\{v_1,w_1\} \subset S$ or $\{v_y,w_z\} \subset S$, then $S$ is not minimal. This covers any case where $V(G) \setminus S$ contains an interior vertex, or when $V(G) \setminus S$ contains two exterior vertices on the same side.

Otherwise, both vertices in $V(G) \setminus S$ are exterior vertices on opposite sides which may or may not be on the same path. Then because $S$ is a Z-set, at least one vertex in $V(G) \setminus S$ is in $\Term(\mathcal{F})$, where $\mathcal{F}$ is the set of forces of $S$. Without loss of generality, say $v_y \in \Term(\mathcal{F})$. Construct a new graph $G'$ of order $m = n+1$ by adding a leaf $v_{y+1}$ adjacent to $v_y$. Note that $S$ is still a Z-set of $G'$, and the order of $S$ in relation to $G'$ is $m-3$. Also by adding a leaf, see that $Z(G')$ is still 2. Therefore we can apply  Proposition \ref{Side1} to $G'$ to show that $\overline{Z}(G') \neq m-3$, which means that $S$ is not minimal in $G'$. This further implies that $S$ is not minimal in $G$. Therefore $\zbar \neq n-2$. Finally, by Proposition \ref{Side1}, $\zbar \neq n-3$, and thus $\zbar < n-3$.
\end{proof}

Thus, we have disproved the existence of a graph $G$ satisfying both $Z(G) = 2$ and $\md(G) > n-6$ for order $n\geq 10$.

\subsection{Bounding $\zbar$ when $Z(G) = 3$}

To disprove the case where $Z(G) = 3$ and $\zbar = n-2$, we must consider the characterizations of a graph on three forcing chains. First, we give a simple observation regarding the planarity of this family.

\begin{observation}\label{observation1}
    If $Z(G)\leq 3$, then $G$ is planar.
\end{observation}

\begin{proof}
    It has been established in \cite[Observation 2.4, Corollary 2.24]{barrett} that every graph $G$ satisfies $\mu(G) \leq M(G) \leq Z(G)$, where $\mu(G)$ is the Colin de Verdière number, $M(G)$ is the maximum nullity, and $Z(G)$ is the zero forcing number. If $Z(G)\leq3$, then $\mu(G) \leq M(G) \leq 3$. By \cite[Theorem 1]{mccarty}, since $\mu(G) \leq 3$, then G is planar.
\end{proof}

Now, we extend the outerplanarity argument to induced subgraphs on forcing chains.

\begin{lemma}\label{characterization2}
    If $G$ is a graph satisfying $Z(G)=3$ and $\mathcal{Z}$ is a minimum zero forcing set of $G$, then the induced subgraph $G'$ on the vertices of two forcing chains of the set $\mathcal{Z}$ satisfies $Z(G')\leq 2$ and is outerplanar.
\end{lemma}

\begin{proof}
    Let $\mathcal{Z}$ be a minimum Z-set of $G$, partitioning $V(G)$ into three forcing chains. Each vertex in $\mathcal{Z}$ is the first vertex of a chain, and every vertex in $V(G)\setminus \mathcal{Z}$ is in exactly one forcing chain. The induced subgraph $G'$ on the set of vertices in any two of these two forcing chains must satisfy $Z(G') \leq 2$, since both $V(G')\cap \mathcal{Z}$ and $V(G')\cap \Term(\mathcal{Z})$ are Z-sets of $G'$. 

    If $Z(G')=1$, then $G'\cong P_{|G'|}$ and $G'$ is outerplanar. Otherwise if $Z(G')=2$, then $G'$ is a graph on two parallel paths and is by definition outerplanar \cite[Theorem 2.3]{row}.
\end{proof}

Using Lemma \ref{characterization2}, we now move towards disproving the existence of a graph $G$ of order $n\geq 17$ satisfying both $Z(G) = 3$ and $\zbar = n-2$. To do this, we first address in Lemma \ref{combined lemma} a specific characterization of a Z-set $S$ that is particularly difficult to show not minimal. Note that we picked $n\geq 17$ as it significantly reduces the number of cases to be considered, but we believe this order could be reduced further. We leave this as an open question in Section \ref{conclusion}.

\begin{lemma}\label{combined lemma}
    Given a connected graph $G$ of order $n \geq 17$ satisfying $Z(G) = 3$, let $\mathcal{Z'}$ be a minimum zero forcing set of $G$ and $\mathcal{F}$ be a set of forces of $\mathcal{Z'}$ that can be partitioned into three maximal forcing chains. If under a zero forcing set $S$ of size $n-2$:

    \begin{itemize}
        \item One of the uncolored vertices is the starting vertex of a forcing chain of $\mathcal{F}$,
        \item the other is in the terminus of $\mathcal{F}$, and
        \item these vertices are in unique forcing chains,
    \end{itemize}

    \noindent then $S$ is not a minimal Z-set.
\end{lemma}

\begin{proof}
    Define $\mathcal{Z'}= \{u_1,v_1,w_1\}$ and Term$(\mathcal{F}) = \{u_x, v_y, w_z\}$, both of which are minimum Z-sets by \cite[Theorem 2.6]{barioli}. Let the forcing chains be labeled $\mathcal{F}_1 = \{u_1,..., u_x\}, \mathcal{F}_2 = \{v_1,..., v_y\}$, and $\mathcal{F}_3 = \{w_1,..., w_z\}$, and note that $\{u_i, v_j, w_k\}$ are pairwise distinct for any $i, j, k$. Without loss of generality, suppose the two uncolored vertices are $u_1$ and $w_z$, meaning that $V(G) \setminus S = \{u_1, w_z\}$. Now consider the induced subgraph $G'$ illustrated in Figure \ref{ppg} over $\mathcal{F}_1 \cup \mathcal{F}_3$, the two forcing chains that contain an uncolored vertex. Observe that $G'$ is outerplanar and $Z(G') \leq 2$ by Lemma \ref{characterization2}. 

\vspace{5pt}

\begin{figure}[H]
  \centering
  \begin{minipage}{0.48\textwidth}
    \centering
    \begin{tikzpicture}[scale=.65, transform shape]
        \node[main node, fill=white] (1) {$u_1$};
        \node[main node] (2) [right = 1cm of 1] {$u_2$};
        \node[main node] (3) [right = 1cm of 2] {$u_3$};
        \node[main node] (4) [right = 1cm of 3] {$u_4$};
        \node[main node] (5) [right = 1cm of 4] {$u_x$};

        \node[main node] (6) [below = 1cm of 1] {$v_1$};
        \node[main node] (7) [below = 1cm of 2] {$v_2$};
        \node[main node] (8) [below = 1cm of 3] {$v_3$};
        \node[main node] (9) [below = 1cm of 4] {$v_y$};
        
        \node[main node] (11) [below = 1cm of 6] {$w_1$};
        \node[main node] (12) [below = 1cm of 7] {$w_2$};
        \node[main node] (13) [below = 1cm of 8] {$w_3$};
        \node[main node] (14) [below = 1cm of 9] {$w_4$};
        \node[main node, fill=white] (15) [right = 1cm of 14] {$w_z$};

        \path[draw,thick]
        (1) edge node {} (2)
        (2) edge node {} (3)
        (3) edge node {} (4)

        (6) edge node {} (7)
        (7) edge node {} (8)

        (11) edge node {} (12)
        (12) edge node {} (13)
        (13) edge node {} (14);

        \path[draw]
        (1) edge[dashed] node {} (6)
        (7) edge[dashed] node {} (2)
        (9) edge[dashed] node {} (4)
        (7) edge[dashed] node {} (14)
        (9) edge[dashed] node {} (15)
        (11) edge[dashed] node {} (4)
        (5) edge[dashed] node {} (14);

        \path (4) -- node[xshift=0cm] {$\dots$} (5);
        \path (8) -- node[xshift=0cm] {$\dots$} (9);
        \path (14) -- node[xshift=0cm] {$\dots$} (15);
    \end{tikzpicture}
  \end{minipage}
  \hfill
  \begin{minipage}{0.48\textwidth}
    \centering
    \begin{tikzpicture}[scale=.65, transform shape]
        \node[main node, fill=white] (1) {$u_1$};
        \node[main node] (2) [right = 1cm of 1] {$u_2$};
        \node[main node] (3) [right = 1cm of 2] {$u_3$};
        \node[main node] (4) [right = 1cm of 3] {$u_4$};
        \node[main node] (5) [right = 1cm of 4] {$u_x$};
        
        \node[main node] (11) [below = 1cm of 1] {$w_1$};
        \node[main node] (12) [below = 1cm of 2] {$w_2$};
        \node[main node] (13) [below = 1cm of 3] {$w_3$};
        \node[main node] (14) [below = 1cm of 4] {$w_4$};
        \node[main node, fill=white] (15) [right = 1cm of 14] {$w_z$};

        \path[draw,thick]
        (1) edge node {} (2)
        (2) edge node {} (3)
        (3) edge node {} (4)

        (11) edge node {} (12)
        (12) edge node {} (13)
        (13) edge node {} (14);

        \path[draw]
        (11) edge[dashed] node {} (4)
        (5) edge[dashed] node {} (14);

        \path (4) -- node[xshift=0cm] {$\dots$} (5);
        \path (14) -- node[xshift=0cm] {$\dots$} (15);
    \end{tikzpicture}
  \end{minipage}

  \caption{A graph $G$ and outerplanar induced subgraph $G'$.}
  \label{ppg}
\end{figure}
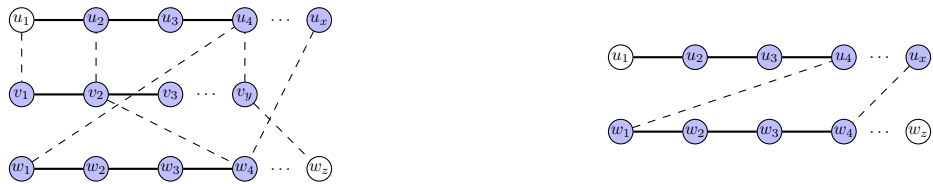

    If the order of $G'$ is $n\geq 10$, we refer to Theorem \ref{cor2} to draw a contradiction by showing that $S \cap V(G')$ is not a minimal Z-set of $G'$, and thus $S$ is not a minimal Z-set of $G$. Otherwise consider the following cases split by the orders, $x$ and $z$ respectively, of the two parallel paths of $G'$, up to isomorphism. Given the specified Z-set $S$ of size $n-2$, we will use outerplanarity and Propositions \ref{reforce}, \ref{nbhd}, and \ref{fort} to show that $S$ is not minimal.
    
    \begin{description}
        \item[Case 1:]  Let $x \geq 3$ and $z \geq 1$ such that $x+z \leq 9$.

        First, if $u_xw_z \notin E(G)$, then we apply Proposition \ref{reforce} to $u_x$. Now assume $u_xw_z \in E(G)$. If there exists some $v_i \neq v_y$ such that $v_i$ is not adjacent to $u_x$, then Proposition \ref{nbhd} is satisfied by choosing $S' = N_c[u_x] \cup \{v_y\}$ and $\mathcal{Z} = \{u_x, v_y, w_z\}$.
        
        Otherwise each $v_i \in \{v_1,..., v_{y-1}\}$ is adjacent to $u_x$. By outerplanarity we know that $v_yu_1 \notin E(G)$, so we consider the edge $v_yw_z$. If $v_yw_z\notin E(G)$, then we apply Proposition \ref{reforce} to $v_y$. Else $v_yw_z\in E(G)$, and Proposition \ref{nbhd} is satisfied by choosing $S' = N_c[v_y] \cup \{u_x\}$ and $\mathcal{Z} = \{u_x, v_y, w_z\}$.

        \item[Case 2:]  Let $x=2$ and $z=2$.
        
        If there exists some $v_i$ adjacent to neither $u_1$ nor $w_2$, then apply Proposition \ref{reforce} to $v_i$. Let the symmetric difference $(A\cup B) \setminus (A \cap B)$ of two sets be defined $A \ominus B$. Now suppose there exists some $v_i \in N(u_1)\ominus N(w_2)$, and assume without loss of generality that $v_i$ is adjacent to $u_1$. Then, Proposition \ref{nbhd} is satisfied by choosing $S' = N_c[v_i] \cup \{v_1, w_1\}$ and $\mathcal{Z} = \{u_1, v_1, w_1\}$.
        
        Otherwise each $v_i \in \{v_1,..., v_y\}$ is adjacent to both $u_1$ and $w_2$. By Proposition \ref{fort}, at least one of $\{u_2, w_1\}$ is adjacent to exactly one uncolored vertex. Assume without loss of generality that this is $u_2$, implying that $u_2w_2\notin E(G)$. By outerplanarity, $v_iu_2\notin E(G)$ for each $v_i\in \{v_1,..., v_{y-1}\}$. Then, Proposition \ref{nbhd} is satisfied by choosing $S' = N_c[u_2] \cup \{v_1, w_1\}$ and $\mathcal{Z} = \{u_1, v_1, w_1\}$.

        \item[Case 3:]  Let $x=2$ and $z=1$.
        
        If there exists some $v_i$ adjacent to neither $u_1$ nor $w_1$, then apply Proposition \ref{reforce} to $v_i$. If there exists some $v_i$ adjacent to $w_1$ but not to $u_1$, then Proposition \ref{nbhd} is satisfied by choosing $S' = N_c[v_i] \cup \{u_2, v_y\}$ and $\mathcal{Z} = \{u_2, v_y, w_1\}$. Otherwise, every $v_i \in \{v_1,..., v_y\}$ is adjacent to $u_1$. 
        
        By outerplanarity, it follows that $v_iu_2\notin E(G)$ for each $v_i \in \{v_1,..., v_{y-1}\}$. First, if there exists some $v_i$ not adjacent to $w_1$, we consider subcases on $u_2w_1$. If $u_2w_1\notin E(G)$, let $S' = S \setminus \{v_j\}$ where we choose some $v_j\in N_c(v_i)\cap \{v_1,..., v_{y-1}\}$. Observe that $S'$ is a Z-set, implying that $S$ is not minimal, since $u_2 \rightarrow u_1$, $v_i \rightarrow v_j$, and finally $w_1$ is forced. If $u_2w_1 \in E(G)$, let $S' = S \setminus \{v_j\}$ where we choose some $v_j\notin N_c[v_i] \cup \{v_y\}$. Observe that $S'$ is a Z-set, again implying that $S$ is not minimal, since $v_i \rightarrow u_1$, $u_2 \rightarrow w_1$, and finally $v_j$ is forced. Otherwise each $v_i \in \{v_1,..., v_y\}$ is adjacent to both $u_1$ and $w_1$ and thus $u_2w_1 \notin E(G_n)$ by Proposition \ref{fort}. Consequently, Proposition \ref{nbhd} is satisfied by choosing $S' = S\setminus \{v_1,..., v_{y-2}\}$ and $\mathcal{Z} = \{u_2, v_y, w_1\}$.

        \item[Case 4:]  Let $x=1$ and $z=1$.
        
        By Proposition \ref{fort}, there exists some $v_i \in N(u_1) \ominus N(w_1)$. Assume without loss of generality that $v_i$ is adjacent to $u_1$. If there exists some $v_j\neq v_i$ adjacent to $w_1$, define $X = \{v_2,..., v_y\} \setminus (N[v_i]\cup N[v_j])$ and observe that Proposition \ref{nbhd} is satisfied by choosing $S' = S \setminus X$ and $\mathcal{Z} = \{u_1, v_1, z_1\}$.
        
        Otherwise $v_kw_1\notin E(G)$ for each $v_k\in \{v_1,..., v_y\}$, and thus $u_1w_1 \in E(G)$ since $G$ is connected. Observe that $S' = S \setminus \{v_j\}$ for some $v_j\notin N_c[v_i]$ is a Z-set, implying that $S$ is not minimal, since $v_i \rightarrow u_1$, $v_{j \pm 1} \rightarrow v_j$, and finally $u_1 \rightarrow w_1$.
    \end{description}

    Thus, $S$ is not minimal.
\end{proof}

Now we will show that if $Z(G) = 3$, then $\zbar < n-2$ for all graphs $G$ of order $n\geq 17$. This will effectively disprove the existence of a graph of order $n\geq 17$ satisfying both $Z(G) = 3$ and $\md(G) > n-6$.

\begin{theorem}\label{Side2}
    If $G$ is a graph of order $n\geq 17$ satisfying $Z(G) = 3$, then $\zbar < n-2$.
\end{theorem}

\begin{proof}
Let $G$ be a graph of order $n\geq 17$ satisfying $Z(G) = 3$. By Proposition \ref{obs}, we know that $\zbar < n-1$. Thus for the sake of contradiction, let $\zbar = n-2$. Consider a minimum Z-set $\mathcal{Z}= \{u_1,v_1,w_1\}$. The graph $G$ can be visualized as three forcing chains of the set $\mathcal{Z}$ of orders $x$, $y$, and $z$ such that $n=x+y+z$. Note also that by \cite[Theorem 2.6]{barioli}, $\{u_x, v_y, w_z\}$ is a minimum Z-set as well. We define $\{u_1, u_x, v_1, v_y, w_1, w_z\}$ to be exterior vertices and $V(G_n)\setminus \{u_1, u_x, v_1, v_y, w_1, w_z\}$ to be interior vertices.

If $G$ is disconnected, then $G$ is either three disjoint paths and $\zbar \leq 6 < n-2$, or $G$ is two connected paths disjoint with another path. In either case, because $|G|\geq 17$, we can take the induced subgraph over two of these paths and obtain a subgraph $G'$ or order $n\geq 10$. Thus, we apply Theorem \ref{cor2} to $G'$ to show that $\overline{Z}(G') < n-2$ and therefore $\zbar < n-2$. Thus, we will only consider connected graphs. Note that, as proven in Lemma \ref{characterization2}, the induced subgraph over any two forcing chains is outerplanar. Consider the following characterizations of $V(G)\setminus S$, where $S$ is a minimal Z-set of maximum size, $|S| = n-2$.

First, if there exists some interior vertex $p\in V(G)\setminus S$, then either ${\{u_1, v_1, w_1\} \subset S}$ or ${\{u_x, v_y, w_z\}\subset S}$. Both of these are zero forcing proper subsets of $S$, contradicting the minimality of $S$. Thus, all interior vertices must be in $S$. In each of the remaining cases, we will also draw a contradiction to show that $S$ is not minimal. Consider the following cases where all interior vertices are in $S$:

    \begin{description}
        \item[Case 1:] Suppose the vertices $V(G)\setminus S$ are on the same end. Without loss of generality, say $V(G) \setminus S = \{u_1, w_1\}$.
        
        Then the Z-set $\{u_x, v_y, w_z\} \subset S$, contradicting $S$ being minimal.

        \item[Case 2:] Suppose the vertices $V(G)\setminus S$ are on opposite ends and on different paths. Without loss of generality, say $V(G) \setminus S = \{u_1, w_z\}$.
        
        By Lemma \ref{combined lemma}, it can be seen that $S$ is not a minimal Z-set of $G$, a contradiction.
        
        \item[Case 3:] Suppose the vertices $V(G)\setminus S$ are on opposite ends of the same path subgraph. Without loss of generality, say $V(G) \setminus S = \{u_1, u_x\}$. 
        
        Consider the induced subgraph $G'$ of the vertices in this path and the remaining path of higher order, say $\{w_1,..., w_z\}$. $G'$ must then have order 10 or greater, because if the first path has order $x\geq2$, then the second must have order $z\geq \big\lceil \frac{17-x}{2}\big\rceil \geq 8$, summing to $|G'| = x + z \geq 10$. We can then refer to Theorem \ref{cor2} to show that $S \cap V(G')$ is not a minimal Z-set of $G'$, and thus $S$ is not a minimal Z-set of $G_n$, a contradiction.
    \end{description}

    Thus by contradiction we conclude that if $G$ is a graph of order $n\geq17$ satisfying $Z(G) = 3$, then $\zbar \neq n-2$. We then conclude that under the constraints of this proof, $\zbar < n-2$.
\end{proof}

Combining our previous results, we now make our main claim that $n-6$ is indeed a sharp upper bound on the minimal zero forcing set discrepancy.

\begin{theorem}\label{one theorem to rule them all}
    For any graph $G$ of order $n\geq 17$, the minimal zero forcing set discrepancy ${\md(G) \leq n-6}$, and this bound is sharp for an infinite family of graphs.
\end{theorem}

\begin{proof}
    Proposition \ref{initial bound}, Theorem \ref{cor2}, and Theorem \ref{Side2} prove together that ${\md(G) \leq n-6}$, and Proposition \ref{n-6 family} shows that this bound is sharp for an infinite family of graphs.
\end{proof}

\section{Maximizing propagation time discrepancy}\label{pd}

We pose another question, similarly structured to that of Section \ref{zbar-z}: What is the largest possible value of the propagation time discrepancy of a graph between minimum Z-sets, $\pd(G)=\PT(G)-\pt(G)$? First, we note the bounds presented in \cite{chilakamarri}.

\begin{corollary}\textnormal{\cite[Theorem 2.5]{chilakamarri}}\label{bounds}
    For any graph $G$,
    \[\PT(G) \leq n-Z(G), \text{ and }\]
    \[\pt(G) \geq \frac{n-Z(G)}{Z(G)}.\]
\end{corollary}

For a minimum Z-set $S$ of a graph $G$, if only one vertex in $V(G)\setminus S$ is forced at each timestep, then propagation time is maximized, having the value $\PT(G) = n-Z(G)$. Assume instead that $|S|$ vertices are forced at each time step, except for the last time step in which fewer than $|S|$ vertices may remain to be forced. Then, the propagation time is minimized, having the value $\pt(G) = \left\lceil \frac{n-Z(G)}{Z(G)}\right\rceil \geq \frac{n-Z(G)}{Z(G)}$. Note the use of the ceiling function, as $\big(n-Z(G)\big) \bmod \big(Z(G)\big)\neq 0$ implies that another time step is needed to fully color the graph. 

We will first define an upper bound on propagation time discrepancy in terms of the order of a graph, and then show that this bound is sharp.

\begin{theorem}\label{pd upper bound}
  For any graph $G$ of order $n$ and propagation time discrepancy ${\pd(G)=\PT(G)-\pt(G)}$, $$\pd(G) \leq \left\lfloor n-2\sqrt{n}+1 \right\rfloor,$$ and this bound is sharp for an infinite family of graphs.
\end{theorem}

\begin{proof}
    Let $G$ be a graph of order $n$. Corollary \ref{bounds} yields
    \begin{align*}
        \pd(G) &= \PT(G) - \pt(G) \\
            &\le \big(n - Z(G)\big) - \left( \frac{n - Z(G)}{Z(G)} \right) \\
            &= n - Z(G) - \frac{n}{Z(G)} + 1.
    \end{align*}

    \vspace{5pt}

    To find the value of $Z(G)$ that maximizes $\pd(G)$ for any $n$, we take the partial derivative of this equation with respect to $Z(G)$, which yields

    \[-1 + \frac{n}{Z(G)^2}.\]
    \vspace{5pt}

    Setting this equal to 0 and solving for $Z(G)$ yields $Z(G) = \sqrt{n}$. Plugging this back into the equation gives us an upper bound for any graph: $\pd(G) \leq n-\sqrt{n}-\frac{n}{\sqrt{n}}+1 = n-2\sqrt{n}+1$. Since $\pd(G)$ is a nonnegative integer, we can apply the floor function to this quantity. Thus the propagation time discrepancy for any graph is bounded by \[\pd(G) \leq \left\lfloor n-2\sqrt{n}+1 \right\rfloor .\]

To show that this upper bound is sharp, we define an infinite family $\mathcal{G} = \{G_n\}$ as follows. Let $G_n=P_n+E$ where $P_n$ is a path of order $n$ with vertices labeled $\{v_1,v_2,v_3,...,v_n\}$ and $E=\{v_{ik}v_{ik+2}\; | \; 1 \leq i < k\}$, where $n=k^2+1$ for an odd integer $k \geq 3$. See Figure \ref{maxpd} for $G_{26}$, the graph for $k=5$. 

    \vspace{5pt}
    
    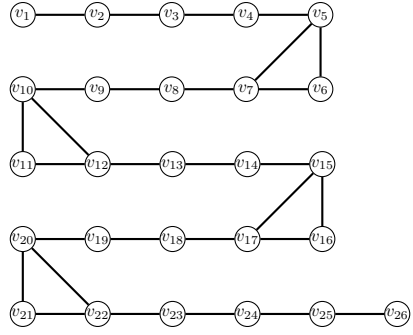
\begin{figure}[H]
    \begin{center}
    \begin{tikzpicture}[scale=.65, transform shape]
    \node[main node, fill=white] (1) {$v_1$};
    \node[main node, fill=white] (2) [right = 1cm of 1] {$v_2$};
    \node[main node, fill=white] (3) [right = 1cm of 2] {$v_3$};
    \node[main node, fill=white] (4) [right = 1cm of 3] {$v_4$};
    \node[main node, fill=white] (5) [right = 1cm of 4] {$v_5$};
    \node[main node, fill=white] (6) [below = 1cm of 5] {$v_6$};
    \node[main node, fill=white] (7) [left = 1cm of 6] {$v_7$};
    \node[main node, fill=white] (8) [left = 1cm of 7] {$v_8$};
    \node[main node, fill=white] (9) [left = 1cm of 8] {$v_9$};
    \node[main node, fill=white] (10) [left = 1cm of 9] {$v_{10}$};
    \node[main node, fill=white] (11) [below = 1cm of 10] {$v_{11}$};
    \node[main node, fill=white] (12) [right = 1cm of 11] {$v_{12}$};
    \node[main node, fill=white] (13) [right = 1cm of 12] {$v_{13}$};
    \node[main node, fill=white] (14) [right = 1cm of 13] {$v_{14}$};
    \node[main node, fill=white] (15) [right = 1cm of 14] {$v_{15}$};
    \node[main node, fill=white] (16) [below = 1cm of 15] {$v_{16}$};
    \node[main node, fill=white] (17) [left = 1cm of 16] {$v_{17}$};
    \node[main node, fill=white] (18) [left = 1cm of 17] {$v_{18}$};
    \node[main node, fill=white] (19) [left = 1cm of 18] {$v_{19}$};
    \node[main node, fill=white] (20) [left = 1cm of 19] {$v_{20}$};
    \node[main node, fill=white] (21) [below = 1cm of 20] {$v_{21}$};
    \node[main node, fill=white] (22) [right = 1cm of 21] {$v_{22}$};
    \node[main node, fill=white] (23) [right = 1cm of 22] {$v_{23}$};
    \node[main node, fill=white] (24) [right = 1cm of 23] {$v_{24}$};
    \node[main node, fill=white] (25) [right = 1cm of 24] {$v_{25}$};
    \node[main node, fill=white] (26) [right = 1cm of 25] {$v_{26}$};

    \path[draw,thick]
    (1) edge node {} (2)
    (2) edge node {} (3)
    (3) edge node {} (4)
    (4) edge node {} (5)
    (5) edge node {} (6)
    (6) edge node {} (7)
    (7) edge node {} (8)
    (8) edge node {} (9)
    (9) edge node {} (10)
    (10) edge node {} (11)
    (11) edge node {} (12)
    (12) edge node {} (13)
    (13) edge node {} (14)
    (14) edge node {} (15)
    (15) edge node {} (16)
    (16) edge node {} (17)
    (17) edge node {} (18)
    (18) edge node {} (19)
    (19) edge node {} (20)
    (20) edge node {} (21)
    (21) edge node {} (22)
    (22) edge node {} (23)
    (23) edge node {} (24)
    (24) edge node {} (25)
    (25) edge node {} (26)
    
    (5) edge node {} (7)
    (10) edge node {} (12)
    (15) edge node {} (17)
    (20) edge node {} (22);
    
    \end{tikzpicture}

    \vspace{5pt}
    \caption{The graph $G_{26}$ satisfying $\pd(G_n) = \left\lfloor n-2\sqrt{n}+1\right\rfloor$.}
    \label{maxpd}

    \end{center}
    \end{figure}

We first show that $Z(G_n) = k$ by using the path cover number of $G_n$. The \textit{path cover number} of a graph, denoted $P(G)$, is the minimum number of vertex-disjoint induced path subgraphs such that all vertices of $G$ are in exactly one path. 
In each graph $G_n = P_n + E$ of the family $\mathcal{G}$, there are $k-1$ vertex-disjoint $K_3$ cycles along $P_n$, each formed by some edge in $E$. For an induced subgraph to be a path, not all vertices of a $K_3$ cycle can be in the subgraph. Thus, each $K_3$ cycle necessitates adding another path in the path cover. Since there are $k-1$ vertex-disjoint $K_3$ cycles in $G_n$, at least $k$ induced path subgraphs are necessary to cover the entire graph. Observe then that the path cover number $P(G_n) \geq k$. By \cite[Theorem 2.13]{hogben2}, $Z(G_n) \geq P(G_n) \geq k$.
Also, we have a Z-set $S = \big\{v_{ik+1} \; | \; i \in \{0,1,2,...,k-1\}\big\}$ of size $k$ so $Z(G_n) \leq k$. Thus $Z(G_n)=k$.
    
Now consider two minimum Z-sets of $G_n$,

    \[\mathcal{Z} = \big\{v_{ik}\; | \; i \in \{2,4,6,...,k-1\}\big\}\cup \big\{v_{ik+1}\; | \; i \in \{0,2,4,...,k-1\}\big\} \text{ and}\] 
    \[\mathcal{Z'} = \big\{v_{ik+1}\; | \; i \in \{0,1,2,...,k-1\}\big\}.\]

The set $\mathcal{Z}$ has propagation time $\pt(G_n)=k$, while $\mathcal{Z'}$ has propagation time $\PT(G_n) = n-k$. These can be verified to be the extreme propagation times using the results of Corollary \ref{bounds} with $n=k^2+1$ and $Z(G)=k$. The difference of these two values results in
    \[\pd(G_n) = (n-k)-k = n-2k = k^2-2k+1 = (k-1)^2.\]

Now observe that \[\big\lfloor n -2\sqrt{n}+1 \big\rfloor = \big\lfloor k^2+1-2\sqrt{k^2+1}+1 \big\rfloor= \big\lfloor(\sqrt{k^2+1}-1)^2 \big\rfloor.\]

It can be shown algebraically that \[(k-1)^2 < (\sqrt{k^2+1}-1)^2 < (k-1)^2+1,\]

\noindent because these three functions of $k$ are increasing and nonintersecting for $k\geq 3$. Thus, we have

\[\pd(G_n)=(k-1)^2 = \big\lfloor(\sqrt{k^2+1}-1)^2 \big\rfloor = \big\lfloor n-2\sqrt{n}+1\big\rfloor.\]

\end{proof}

\subsection{More attainable propagation time discrepancy values}

Similar to Corollary \ref{allmd}, we will also show that for any nonnegative integer $x$, there exists a graph $G_n$ such that $\pd(G_n)= x$. To do so, we first present another infinite family of graphs.

\begin{proposition}\label{spiderprop}
    There exists an infinite family of graphs $\mathcal{G} = \{G_n\}$ such that $\pd(G_n) = \left\lceil  \frac{n}{2} \right\rceil -3$ for all $n \geq 5$.
\end{proposition}

\begin{proof}
    Let $G_n$ be the spider graph with three legs of length $\{1, \left\lfloor \frac{n}{2}\right\rfloor -1, \left\lceil \frac{n}{2}\right\rceil -1\}$ for $n \geq 5$, as seen in Figure \ref{spider}. Let the center vertex be $c$ and enumerate the vertices of each leg as $\{u_1\}$, $\{v_1,..., v_y\}$, $\{w_1,..., w_z\}$, where $y = \left\lfloor \frac{n}{2}\right\rfloor -1$ and $z = \left\lceil \frac{n}{2}\right\rceil -1$. This infinite family $\mathcal{G}$ was adapted from a similar graph in \cite[Figure 3]{hogben}.

    \begin{figure}[H]
\begin{center}
    \begin{tikzpicture}[scale=.65, transform shape]
    \node[main node, fill=white] (1) {$w_z$};
    \node[main node, fill=white] (2) [right = 1cm of 1] {$w_2$};
    \node[main node, fill=white] (3) [right = 1cm of 2] {$w_1$};
    \node[main node, fill=white] (4) [right = 1cm of 3] {$c$};
    \node[main node, fill=white] (5) [right = 1cm of 4] {$v_1$};
    \node[main node, fill=white] (6) [right = 1cm of 5] {$v_y$};
    \node[main node, fill=white] (7) [above = 1cm of 4] {$u_1$};

    \path[draw,thick]
    (2) edge node {} (3)
    (3) edge node {} (4)
    (4) edge node {} (5)
    (4) edge node {} (7);
    
    \path (1) -- node[xshift=0cm] {$\dots$} (2);
    \path (5) -- node[xshift=0cm] {$\dots$} (6);
    
    \end{tikzpicture}

    \vspace{5pt}
    \caption{A spider graph $G_n$ satisfying $\pd(G) = \left\lceil  \frac{n}{2} \right\rceil -3$.}
    \label{spider}

\end{center}
\end{figure}
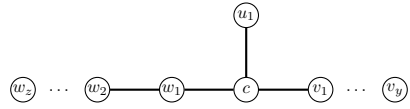
    
    First observe that $Z(G_n) = 2$, since $\{u_1, w_z\}$ is a Z-set and $G_n \ncong P_n$. We present the following table to verify the maximum and minimum propagation times over all the minimum zero forcing sets. Each zero forcing set either needs to be two leaves, or a leaf and a vertex adjacent to the center such that these two vertices are on separate legs. Thus, the following list is comprehensive.

    \begin{table}[H]
    \centering
    \caption{Minimum Z-sets and their propagation times.}
    \vspace{8pt}
    \label{allsets}
    \begin{tabular}{l c}
        \toprule
        \textbf{Z-set} & \textbf{Propagation Time} \\
        \midrule
        $\{u_1, w_z\}$ & $n-3$ \vspace{3pt}\\
        $\{u_1, v_y\}$ & $n-3$ \vspace{3pt}\\
        $\{v_y, w_z\}$ & $\left\lfloor \frac{n}{2} \right\rfloor$ \vspace{3pt}\\
        $\{u_1, w_1\}$ & $\left\lfloor \frac{n}{2} \right\rfloor$ \vspace{3pt}\\
        $\{u_1, v_1\}$ & $\left\lceil \frac{n}{2} \right\rceil$ \vspace{3pt}\\
        $\{w_1, v_y\}$ & $n-3$\vspace{3pt}\\
        $\{v_1, w_z\}$ & $n-3$ \\
        \bottomrule
    \end{tabular}
    \end{table}

    By considering all Z-sets of size 2, we see that $\pt(G_n) = \left\lfloor \frac{n}{2} \right\rfloor$ and $\PT(G_n) = n-3$. The difference results in $\pd(G_n) = \left\lceil \frac{n}{2} \right\rceil -3$.
\end{proof}

The next result follows, using the same family of graphs.

\begin{corollary}\label{allpd}
    For any nonnegative integer $x$, there exists a graph $G_n$ for which $\pd(G_n) = x$.
\end{corollary}

\begin{proof}
    Let $G_n$ be a graph in the family $\mathcal{G}$ from Proposition \ref{spiderprop}, and let $n = 2x+6$. Then \[\pd(G_n) = \left\lceil \frac{2x+6}{2}  \right\rceil -3 = x.\qedhere\]
\end{proof}

\section{Conclusion and Open Problems}\label{conclusion}

In this paper, we studied the maximum discrepancy between various zero forcing parameters. We began in Section \ref{zbar-z} by maximizing the discrepancy between the cardinality of minimal zero forcing sets, introducing an infinite family where $\md(G_n)=n-6$ in Proposition \ref{n-6 family}. This leads us to ask the following questions:

\begin{question}
    Are there other infinite families of graphs satisfying $\md(G_n) = n-6$ besides the family presented in Proposition \ref{n-6 family}?
\end{question}

\begin{question}
    What structural characteristics limit $\md(G)$ to $n-6$ for graphs of order $n\geq17$?
\end{question}

In the beginning of Section \ref{zbar-z}, we noted that there are certain graphs of small order that achieve $\md(G) > n-6$. Our investigation was primarily regarding arbitrarily large order graphs that produce a large discrepancy, specifically of order $n \geq 17$, but we pose the following question:
 
\begin{question}
    What is the largest graph order for which $\md(G) > n-6$?
\end{question}

We have already presented $P_6$ as a trivial example that satisfies $\md(P_6) = n-5$. Regarding graphs of larger order, we see that Proposition \ref{initial bound} implies that all graphs of order $n \geq 7$ satisfy $\md(G) \leq n-4$. For Theorems \ref{cor2} and \ref{Side2}, we bounded the order at some $n \geq 10$ in order to reduce the number of edge cases that needed to be investigated. However, we do suspect that the same result could be proven for $n\geq7$. Thus, we provide a conjecture but leave the question open. 

\begin{conjecture}
    If $G$ is a graph of order $n\geq 7$, then $\md(G) \leq n-6$.
\end{conjecture}

Finally, in Section \ref{pd}, we focused on propagation time discrepancy, finding a universal upper bound on $\pd(G)$. In Theorem \ref{pd upper bound} we also showed that this bound is sharp by introducing an infinite family of graphs that achieves this bound. This leads us to ask the following question about similar families of graphs.

\begin{question}
    Are there other infinite families of graphs satisfying $\pd(G_n)= \left\lfloor n-2\sqrt{n}+1 \right\rfloor$ besides the family presented in Theorem~\ref{pd upper bound}?
\end{question}

We also note that we considered applying the strategies used in this paper to maximize a discrepancy related to the throttling number. The \textit{throttling number} of a graph, denoted $\operatorname{th}(G)$, minimizes the sum of the size of a zero forcing set and its corresponding propagation time. But, defining an upper value for the throttling number, or $\operatorname{TH}(G)$ which instead maximizes the sum, will always result in a value of $n$ which can be achieved by choosing the zero forcing set $S = V(G)$.

Additionally, this paper only considered standard zero forcing, which naturally prompts the following question.

\begin{question}
    Do these bounds on $\md(G)$ and $\pd(G)$ extend to skew forcing? To positive semidefinite forcing?
\end{question}

While certainly not comprehensive, these open questions suggest various avenues for future research, particularly regarding the structure of graphs that achieve a maximum discrepancy. We hope the results presented here are a helpful baseline towards this end, and that the methods developed might be useful in studying these discrepancies under non-standard color change rules.

\section*{Acknowledgments}

This research was conducted as part of the Wheaton College Summer Research Program, with generous support from the Wheaton College Alumni Association and the Wheaton College Provost's Office. We also thank Daniel McKay for his helpful conversations that greatly improved this paper.


\begin{thebibliography}{99}

\bibitem[1]{AIM} AIM Special Work Group, \textit{Zero forcing sets and the minimum rank of graphs}, Linear Algebra Appl. \textbf{428} (2008), no. 7, 1628-1648.

\bibitem[2]{alameda} J. S. Alameda, S. Dillman, and F. Kenter, \textit{Leaky forcing: a new variation of zero forcing}, Australas. J. Combin. \textbf{88} (2024), no. 3, 308–326.

\bibitem[3]{barrett} F. Barioli, W. Barrett, S. Fallat, H. T. Hall, L. Hogben, B. Shader, P. van den Driessche, and H. van der Holst, \textit{Parameters Related to Tree-Width, Zero Forcing, and Maximum  Nullity of a Graph} J. Graph Theory \textbf{72} (2013), no. 2, 146-177.

\bibitem[4]{barioli} F. Barioli, W. Barrett, S. Fallat, H. T. Hall, L. Hogben, B. Shader, P. van den Driessche, and H. van der Holst, \textit{Zero forcing parameters and minimum rank problems} Linear Algebra Appl. \textbf{433} (2010), no. 2, 401-411.

\bibitem[5]{brimkov} B. Brimkov and J. Carlson, \textit{Minimal zero forcing sets}, Australas. J. Combin. \textbf{90} (2024), no. 3, 363-377.

\bibitem[6]{burgarth} D. Burgarth and V. Giovannetti. \textit{Full control by locally induced relaxation}. Phys.
Rev. Lett. \textbf{99} (2007), no. 10, 100501.

\bibitem[7]{butler} S. Butler and M. Young, \textit{Throttling zero forcing propagation speed on graphs}, Australas. J. Combin. \textbf{57} (2013), 65-71.

\bibitem[8]{chilakamarri} K. B. Chilakamarri, N. Dean, C. X. Kang, and E. Yi. \textit{Iteration index of a zero forcing set in a graph}, Bull. Inst. Combin. Appl. \textbf{64} (2012), 57-72.

\bibitem[9]{fast} C. C. Fast and I. V. Hicks. \textit{Effects of vertex degrees on the zero-forcing number and propagation time of a graph}, Discrete Appl. Math. \textbf{250} (2018), 215-226.

\bibitem[10]{haynes} T. Haynes, S. Hedetniemi, S. Hedetniemi, and M. Henning, \textit{Domination in graphs applied to electric power networks}, SIAM J. Discrete Math. \textbf{15} (2002), no. 4, 519-529.

\bibitem[11]{hogben} L. Hogben, M. Huynh, N. Kingsley, S. Meyer, S. Walker, and M. Young, \textit{Propagation time for zero forcing on a graph}, Discrete Appl. Math. \textbf{160} (2012), no. 13-14, 1994–2005.

\bibitem[12]{hogben2} L. Hogben, \textit{Minimum rank problems}, Linear Algebra Appl. \textbf{432} (2010), no. 8, 1961-1974.

\bibitem[13]{ima} IMA-ISU research group on minimum rank, \textit{Minimum rank of skew-symmetric matrices described by a graph}, Linear Algebra Appl. \textbf{432} (2010), no. 10, 2457-2472.

\bibitem[14]{johnson} C. R. Johnson, R. Loewy, and P. A. Smith, \textit{The graphs for which the maximum multiplicity of an eigenvalue is two}, Linear and Multilinear Algebra \textbf{57} (2009), 713–736.

\bibitem[15]{mccarty} R. McCarty, \textit{The Extremal Function and Colin de Verdière Graph Parameter}, Electron. J. Combin. \textbf{25} (2018), no. 2, P2.32.

\bibitem[16]{row} D. Row, \textit{A technique for computing the zero forcing number of a graph with a cut-vertex}, Linear Algebra Appl. \textbf{436} (2012), no. 11, 4423-4432.

\bibitem[17]{yang} B. Yang, \textit{Fast-mixed searching and related problems on graphs}, Theoret. Comput. Sci. \textbf{507} (2013), 100-113.


\end{thebibliography}
\end{document}